\documentclass[reqno]{amsart}
\usepackage[colorlinks=true, pdfstartview=FitV, linkcolor=blue, citecolor=blue]{hyperref}

\usepackage{amssymb,amsmath,amsfonts}
\usepackage{bm}
\usepackage{enumitem}
\usepackage{array}

\usepackage{graphicx}

\newcommand{\N}{\mathbb{N}}
\newcommand{\R}{\mathbb{R}}
\newcommand{\C}{\mathbb{C}}
\newcommand{\Z}{\mathbb{Z}}
\newcommand{\Q}{\mathbb{Q}}

\newcommand{\rar}{\rightarrow}
\newcommand{\mc}{\mathcal}
\newcommand{\mr}{\mathrm}

\newtheorem{theorem}{Theorem}
\newtheorem{lemma}[theorem]{Lemma}

\newtheorem{proposition}[theorem]{Proposition}

\newcommand{\vL}{\varLambda}
\newcommand{\cL}{\mathcal{L}}

\newcommand{\oplam}{\mbox{\Large $\curlywedge$}}

\allowdisplaybreaks

\begin{document}
	
	\title{The Gauss circle problem for Penrose tilings}
	
	\author{Alan Haynes, Christopher Lutsko}
	
	

	\begin{abstract}
	Let $B_R$ denote the closed Euclidean ball of radius $R$ in the plane. In this paper we prove that, if  $V$ is the set of vertices of any unit length rhombic Penrose tiling then, for $R\ge 2$,
	\[\#(V\cap B_R)=\pi C_P R^2 + O(R^{2/3}(\log R)^{2/3}),\]
	where $C_P\approx 1.231$ is a constant.
	\end{abstract}
	
	\maketitle

\section{Introduction}\label{sec.Intro}
Aperiodic tilings of Euclidean space, a topic once regarded as an esoteric pursuit, are now widely recognized as a subject of central interest. Beginning with work of Hao Wang and Robert Berger in the 1960's studying a relationship between aperiodic tilings and the Halting problem for Turing machines, results surrounding aperiodic tilings have had substantial impacts in the last century within mathematics, computer science, the natural sciences, and even the liberal arts. Excellent summaries of many of these results, up until about 2010, can be found in \cite[Chapter 1]{Gard1997}, \cite[Chapters 1, 2]{Sene1995}, \cite[Forward by Roger Penrose]{BaakGrim2013}, and \cite[Chapter 1]{BaakGrim2013}. There has also been significant and highly publicized progress in this subject in the last decade.

One landmark, which has also become an inspiration in art and architecture, was Roger Penrose's discovery of a pair of polygonal tiles which could tile the plane, but could only do so aperiodically. Penrose published a first version of his tilings in 1974 \cite{Penr1974} (see also \cite{Penr1979}). Soon after, they were introduced to the general public in an article by Martin Gardner in the Scientific American \cite{Gard1977}.

\begin{figure}[h]
	\caption{Tiles used to construct rhombic Penrose tilings.\\}\label{fig.PenTiles}
	\centering
	\vspace*{-30bp}
	\includegraphics[width=0.75\textwidth]{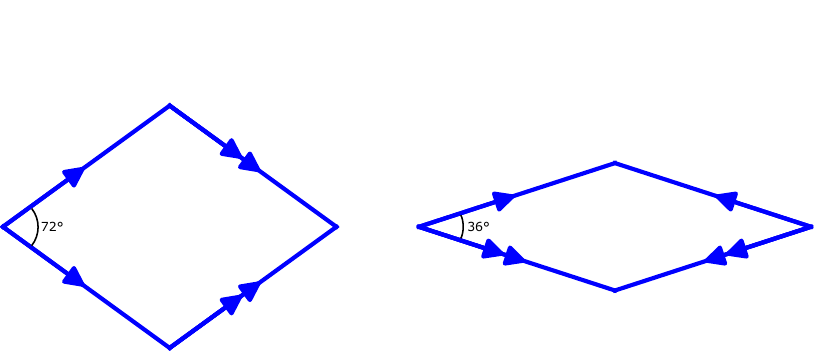}
\end{figure}

There are several related collections of tilings which are referred to as Penrose tilings. In this article we will consider \textit{rhombic Penrose tilings} (RPTs), which are tilings of the plane by copies of two rhombuses with angles $72^\circ$ and $36^\circ$, pictured in Figure \ref{fig.PenTiles}. The rhombuses are allowed to be rotated, but they must fit together edge to edge, and the matching rules indicated by the arrows must also be enforced. A small patch of an RPT is pictured in Figure \ref{fig.PenPatch} (see also \cite[Figures 1.2, 6.44]{BaakGrim2013}). When the rhombuses are chosen to have side length 1, we call the resulting tilings  \textit{unit length RPTs}.

The main problem that we wish to investigate in this paper is that of estimating the number of vertices of an RPT which lie in $B_R$, the closed Euclidean ball of radius $R$, centered at some point in the plane which we fix to be the origin. RPTs can be defined using primitive inflation rules (see \cite[Section 6.2]{BaakGrim2013}), which guarantee existence of limiting frequencies of tile types in large regions. Although it is not obvious, it is not too difficult to deduce (see for example \cite[Remark 5]{SchmTrev2018}) that if $V$ is the vertex set of a unit length RPT then, for $R\ge 1$,
\begin{equation}\label{eqn.PenVertTrivEst}
	\#(V\cap B_R)=\pi C_P R^2 + O(R),
\end{equation}
where
\[C_P=\frac{\phi}{5}\sqrt{10+2\sqrt{5}}\approx 1.231,\]
with $\phi=(1+\sqrt{5})/2$. The error term on the right hand side of \eqref{eqn.PenVertTrivEst} comes from counting tiles which intersect the boundary of $B_R$.
Since there are in fact $\gg R$ such tiles, it may appear that this error term is essentially best possible (cf. results of Solomon from \cite{Solo2011} on counting tiles in \textit{supertiles} of primitive substitutions). However, we will show that for this special class of tilings, one can actually do better.
\begin{theorem}\label{thm.Pen}
Let $V$ be the set of vertices of any unit length RPT. Then, for $R\ge 2$,
\[\#(V\cap B_R)=\pi C_P R^2 + O(R^{2/3}(\log R)^{2/3}).\]
\end{theorem}
Readers who are familiar with the analogous counting problem for integer lattice points in the plane, called the Gauss circle problem, may at this point recognize some similarities. Let us write
\[\#(\Z^2\cap B_R)=\pi R^2 + E(R).\]
Gauss himself observed that $|E(R)|\le 2\sqrt{2}\pi R$, and he conjectured that it should be possible to improve the exponent on $R$  to obtain an error term of $\ll R^{1/2}$. The next substantial result was published in 1906 by Sierpi\'nski \cite{Sier1906}, who showed that $E(R)\ll R^{2/3}$ (see also \cite{Schi1972}). Subsequent improvements have been made by many authors (the introduction of \cite{Huxl2003} contains a list), and currently the best known upper bound, due to Martin Huxley \cite{Huxl2003}, is $E(R)\ll R^{131/208}(\log R)^\tau$, with $\tau=18627/8320$.

The important link which allows us to employ techniques used in the Gauss circle problem in order to study vertices of Penrose tilings is a remarkable observation of de Bruijn, that the vertex set of any RPT can be obtained from a cut and project set, by projecting `slices' of lattice points in a higher dimensional space onto a 2 dimensional subspace. We describe this connection in detail in Section \ref{sec.ModelSets}. With this as a general framework, we are able in Sections \ref{sec.Fourier}-\ref{sec.PenProof} to bring in tools from Fourier analysis and to approach our point counting problem through the lens of the Poisson summation formula (PSF). Results of a similar flavor, estimating discrepancies of points in growing patches of cut and projects sets, have recently been explored by Koivusalo and Lagac{\'e} in \cite{KoivLagacBjorHart2025}, and by R\"uhr, Smilansky, and Weiss in \cite{RuhrSmilWeis2024}.

We remark that our proof, although structurally similar, is not a na\"ive application of known methods (e.g. Huxley's Fourier analytic approach to lattice point counting in regions with smooth boundary, using Van der Corput's method). Here, there are two new difficulties which are not encountered in the classical versions of the Gauss circle problem. The first is that, for the cut and project sets that we consider, the problem amounts to counting points in a lattice in a $4$-dimensional cylinder which is only growing radially. Since the height of the cylinder stays fixed, the smoothing necessary to apply the PSF must shrink faster in the direction of the height than in the radial direction. Secondly, the lattices from which we derive RPTs are irrational. This has the effect that, depending on how well approximable the lattice points are by rational points, smoothing our regions may pick up points which are unusually close to their boundaries. This is a real phenomenon, which can potentially cause a blow up in the Fourier coefficients after using the PSF. It is one of the main issues that we must confront using the specific geometric and algebraic data which define RPTs.


\begin{figure}[h]
	\caption{A patch of an RPT constructed using the cut and project method detailed in Section \ref{sec.ModelSets}. The 4 colors of vertices are translates of the cut and project sets produced using the different windows.\\}\label{fig.PenPatch}
	\centering
	\includegraphics[width=0.75\textwidth]{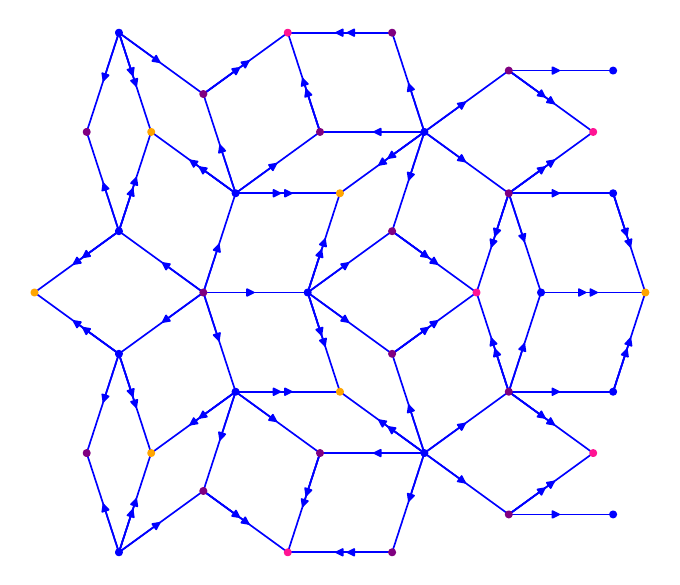}
\end{figure}

Finally, there is the question of how these results could be generalized. Several parts of the proofs below can be adapted quite well to deal with much wider classes of cut and project sets. However, there are some crucial details in the proof, in particular those contained in Section \ref{sec.PenProof}, which form a narrow gate that any more general result would need to pass through. The cut and project sets used to define RPTs are extremely special, both in the algebraic properties of the lattice involved and in the exact shape, symmetry, and orientation of the windows involved, with respect to the specific underlying lattice. Any attempt to modify the lattice, or change the windows (even by a small rotation) will destroy important features of the argument and make it difficult or even impossible to control the sizes of the Fourier coefficients involved.

On the other hand, it seems likely that the arguments given here should apply equally well, with the appropriate modifications, to vertex sets of Ammann-Beenker tilings (see \cite[Section 10.4]{GrunShep1987} and \cite{Been1982}). We leave it to the interested reader to work out the details.

%

\section{Vertices of Penrose tilings via cut and project sets}\label{sec.ModelSets}
In his 1981 paper \cite{deBr1981}, de Bruijn developed the theory of `pentagrids' to explain how vertex sets of RPTs can be constructed in a systematic mathematical way. There are several formulations which are essentially equivalent to the pentagrid construction. The one that we will use in this paper is the construction via cut and project sets. This is explained in de Bruijn's paper \cite[Section 7]{deBr1981}, and significant further investigation can be found in work of G\"ahler and Rhyner \cite{GahlRhyn1986} and Robinson \cite{Robi1996}. To keep our presentation consistent with modern notation and conventions regarding these sets, we will closely follow the exposition of this topic given in Baake and Grimm's book, Aperiodic Order, Vol. 1 (see \cite[Example 7.11]{BaakGrim2013}).

Suppose that $k>d\ge 1$ are integers. Our construction takes place in $\R^k$, which we call the \textit{total space}. Write as $\R^k=\R^d {\times}  \R^{k-d}$, with projections
$\pi\colon \R^k\rar\R^d$ onto the first $d$ coordinates (called the
\textit{physical space}), and $\pi^{}_{\mathrm{int}}\colon \R^k\rar\R^{k-d}$ onto the last $(k-d)$
coordinates (called the \textit{internal space}). We refer to the physical and internal spaces as $G$ and $H$, respectively. Suppose that $\cL\subset\R^k$ is a lattice (i.e. a discrete subgroup of $\R^k$ for which the quotient $\R^k/\cL$ is compact), and further suppose that the map $\pi|^{}_{\cL}$ sending points of $\cL$ to $L:= \pi(\cL)$ is injective, and that the set $L^{\star}:=\pi^{}_{\mathrm{int}}(\cL)$ is dense in $H$. When these conditions are satisfied, the triple $(G,H,\cL)$ is referred to as a \emph{cut and project scheme} (CPS).

The injectivity of $\pi|_\cL$ allows us to define a map $\star\colon L\rar L^{\star}$, called the \emph{star map} of the CPS, by
\begin{equation*}
	x\, \mapsto\, x^{\star} \, = \, \pi^{}_{\mathrm{int}}
	\bigl( \pi^{-1}(x)\cap\cL \bigr) .
\end{equation*}
The domain of this map can be extended in a canonical way to $\Q L$, but not to all of $G$. However, when considering points in $\R^k=\R^d {\times} \R^{k-d}$, it will also be convenient in what follows to write an element $w\in\R^k$ as $w=(x,x^{\star})$ with $x=\pi(w)\in\R^d$ and $x^{\star}=\pi^{}_{\mathrm{int}}(w)\in\R^{k-d}$. Although technically this is a different use of the symbol $\star$, it is consistent with our other use when $w\in\Q\cL$.

The situation so far is summarized in the following diagram.
\begin{equation}\label{eqn.CPSDiagram}
	\renewcommand{\arraystretch}{1.2}\begin{array}{r@{}ccccc@{}l}
		& G & \xleftarrow{\,\;\;\pi\;\;\,} & G \times H & 
		\xrightarrow{\;\pi^{}_{\mathrm{int}\;}} & H & \\
		& \cup & & \cup & & \cup & \hspace*{-1ex} 
		\raisebox{1pt}{\text{\footnotesize dense}} \\
		& \pi(\cL) & \xleftarrow{\; 1-1 \;} & \cL & 
		\xrightarrow{\; \hphantom{1-1} \;} & \pi^{}_{\mathrm{int}}(\cL) & \\
		& \| & & & & \| & \\
		& L & \multicolumn{3}{c}{\xrightarrow{\qquad\qquad\;\;\,\star\,
				\;\;\qquad\qquad}} 
		&  {L_{}}^{\star} & \\
	\end{array}\renewcommand{\arraystretch}{1}
\end{equation}
Given a CPS $(G,H,\cL)$ and a set $W\subseteq H$ (called the
\emph{window}), we define the \emph{cut and project set} associated to $W$ as
\begin{equation*}
	\vL \, = \, \oplam(W)\, = \, 
	\{x\in L : x^{\star}\in W \}  .
\end{equation*}
Note that $\oplam(W)$ is a point set in $G$. In our applications, the window $W$ will be a closed and bounded region with non-empty interior in $H$. This ensures that $\oplam(W)$ is a Delone set, i.e. that it is both uniformly discrete and relatively dense in $G$.

In later sections when we work with Fourier transforms, we will also need to refer to the dual lattice of $\cL$, which we denote by $\cL^{*}$, and to its \textit{Fourier module}, defined by $L^{\circledast}:=\pi(\cL^{*})\subset G$. It can be shown that when $\mc{L}$ is part of a CPS as in \eqref{eqn.CPSDiagram}, the restriction $\pi|_{\cL^*}$ of $\pi$ to $\cL^*$ is injective \cite[Proposition 2.2]{BaakHayn2026}.

Now let $K=\Q(\zeta_5)$, where $\zeta_5$ is a primitive $5$th root of unity. The four embeddings of $K$ into $\C$ are the field homomorphisms $\sigma_1$ and $\sigma_2$ determined by $\sigma_1(\zeta_5)=\zeta_5$ and $\sigma_1(\zeta_5)=\zeta_5^2$, and their complex conjugates. We identify $\C$ with $\R^2$ and define the \textit{Minkowski embedding} $\sigma:K\rar\R^4$ by $\sigma(\alpha)=(\sigma_1(\alpha),\sigma_2(\alpha))$. From basic algebraic number theory (see \cite[Chapter 1]{Swinn2001}), the Minkowski embedding is an injective map, and the image under $\sigma$ of any $\Z$-module of maximal rank in $K$ is a lattice in $\R^4$.

We fix a CPS by selecting $G=H=\R^2$ (still identifying these, when convenient, with $\C$), and $\cL=\sigma \left((1-\zeta_5)\Z[\zeta_5]\right)$. In this case we have $L=(1-\zeta_5)\Z[\zeta_5]$ and $L^\star=(1-\zeta_5^2)\Z[\zeta_5],$ and it is also not difficult to compute that that $\cL^*\subseteq\frac{1}{25}\cL$. The density of $L^\star$ in $H$, although not obvious, can be established directly using arguments from the geometry of numbers. However, we omit the proof since it is well known and it also follows as a special case of the Strong Approximation Theorem (see \cite[Chapter II, Section 15]{CassFroh1967}).

To define windows, let $P$ denote the convex hull of the fifth roots of unity in $H$, let $\tau=\frac{1+\sqrt{5}}{2}$, and for each $\omega\in\C$ let
\begin{align*}
	W_{1,\omega}&=P-1+\omega,\\
	W_{2,\omega}&=-\tau P-2+\omega,\\
	W_{3,\omega}&=\tau P-3+\omega,\quad\text{and}\\
	W_{4,\omega}&=-P-4+\omega.
\end{align*}
Then define cut and project sets
\[\Lambda_{m,\omega}=\oplam(W_{m,\omega}),\]
and finally set
\[\Lambda_\omega=\bigcup_{m=1}^4\left(\Lambda_{m,\omega}+m\right)=\bigcup_{m=1}^4\Lambda_{m,\omega+m}.\]
Here we distinguish two cases, which depend on the choice of $\omega$. If $L^*$ intersects the boundary of $W_{m,\omega}$, for some $m$, then we say that $\Lambda_\omega$ is \textit{singular}. Otherwise $\Lambda_\omega$ is \textit{non-singular}. It is clear that, for almost all choices of $\omega$, the set $\Lambda_\omega$ is non-singular.

The work of de Bruijn, et al., mentioned at the beginning of this section (see references listed above), implies the following result.
\begin{theorem}\label{thm.PenVerts}
Every non-singular point set $\Lambda_\omega$, constructed as above, is the set of vertices of a unit length rhombic Penrose tiling. Furthermore, the vertex set of any unit length rhombic Penrose tiling can be obtained, after rotation, as the limit in an appropriate topology (e.g. the Chabauty-Fell topology) of a sequence of non-singular point sets $\Lambda_\omega$.
\end{theorem}
An example of a RPT obtained using the cut and project method described above is shown in Figure \ref{fig.PenPatch}. Theorem \ref{thm.PenVerts} will allow us to derive Theorem \ref{thm.Pen} from the following result.
\begin{theorem}\label{thm.PenCPVersion}
For the cut and project sets $\Lambda_{m,\omega}$ defined above, there is a constant $C>0$ with the property that, for any $\omega\in\C$, any integer $1\le m\le 4$, and any $R\ge 2$,
\[\left|\#\left(\Lambda_{m,\omega}\cap B_R\right)-\pi R^2\mr{dens}(\cL)\left|W_{m,\omega}\right|\right|\le CR^{2/3}(\log R)^{2/3},\]
where $|W_{m,\omega}|$ denotes the area of $W_{m,\omega}$, and $\mr{dens}(\cL)$ is the reciprocal of the volume of a fundamental domain for $\R^4/\cL$.
\end{theorem}
An important point is that the error term in this theorem is uniform in $\omega$, which allows us to take limits of non-singular point sets, to account for all possible vertex sets of RPTs, in order to establish Theorem \ref{thm.Pen}. There are a few other small details that have to be ironed out, but they are not difficult to justify, so we postpone their discussion until the end of Section \ref{sec.PenProof}. The majority of the remainder of the paper will be focused on proving Theorem \ref{thm.PenCPVersion}.

\section{Results from Fourier analysis}\label{sec.Fourier}
In this section we will summarize notation and background results in Fourier analysis that we need for our proofs. However, we assume that the reader is familiar with basic properties of Fourier transforms as presented, for example, in \cite[Chapter 1]{SteiWeis1971}.

For $z\in\C$ let $e(z)=e^{2\pi i z}$. We define the \emph{Fourier transform} of a complex-valued function $\phi\in L^1(\R^s)$ to be the function
$\widehat{\phi}\colon\R^s\rar\C$ given by
\[
\widehat{\phi}(y) \, =
\int_{\R^s_{\vphantom{t}}}e(-x\cdot y) \, \phi(x ) \mr{d} x.
\]
The \emph{convolution} of $\phi^{}_1,\phi^{}_2\in L^1(\R^s)$
is the function $\phi^{}_1  \ast \phi^{}_2\in L^1(\R^s)$ defined
by
\[
\bigl( \phi^{}_1  \ast \phi^{}_2 \bigr) (y) \, =
\int_{\R^s_{\vphantom{t}}}\phi^{}_1(x)\,
\phi^{}_2(y-x) \mr{d} x,
\]
and we have that
\begin{equation*}
	\widehat{\phi^{\rule{0pt}{2pt}}_1\!\ast
		\phi^{}_2} (y) \, = \,
	\widehat{\phi^{}_1}(y)\,\widehat{\phi^{}_2}(y)
	\qquad\text{for all}~y\in\R^s .
\end{equation*}
We will use the following version of the Poisson summation formula, a proof of which can be found in \cite[Proposition 2.3]{BaakHayn2026}.
\begin{proposition}\label{prop.PSF}
	Let\/ $\cL\subset\R^s$ be a lattice, with dual lattice\/
	$\cL^*$. Suppose that\/ $\phi$ is a continuous function with compact
	support and that
	\begin{equation}\label{eqn.PSFFourDec}
		\sum_{\xi\in\cL^\ast_{\vphantom{t}}} \bigl| \widehat{\phi}(\xi) \bigr|
		\, < \, \infty.
	\end{equation}
	Then, for all\/ $y\in\R^s$, we have the identity
	\begin{equation*} 
		\sum_{\ell\in\cL}\phi(y+\ell) \, = \, 
		\mathrm{dens} (\cL) \!\sum_{\xi\in\cL^*_{\vphantom{t}}}\!
		\widehat{\phi}(\xi) \, e(\xi\cdot y),
	\end{equation*}
	where $\mr{dens}(\cL)$ is the reciprocal of the volume of a fundamental domain for $\R^{s}/\cL$.
\end{proposition}
In many cases, we would like to apply the proposition above to functions $\phi$ which are indicator functions of sets in $\mathbb{R}^2$. The obvious issues are that such functions are not continuous and, moreover, in interesting cases their Fourier transforms will not decay rapidly enough for \eqref{eqn.PSFFourDec} to be satisfied. There are several ways of dealing with these issues. The one that we will use is to first form the convolution of our indicator functions with non-negative smooth functions with small compact support. We define $\psi:\R^2\rar\R$ by
\[\psi(x)=\begin{cases}
	\mathrm{exp}\left(\frac{-1}{1-4|x|^2}\right)&\text{if}\quad |x|<1/2,\\
	0&\text{if}\quad |x|\ge 1/2.
\end{cases}\]
Since $\psi\in C_c^\infty(\R^2)$, it is a Schwartz class function, and therefore so is its Fourier transform \cite[Chapter I, Theorem 3.2]{SteiWeis1971}. It follows from this that, for every $N\in\N$,
\[|\widehat{\psi}(y)|\ll_N\frac{1}{(1+|y|^2)^N},\quad\text{for}~ y\in\R^2.\]
Next, for $\epsilon>0$, we define $\psi_\epsilon:\R^2\rar\R$ by
\begin{equation}\label{eqn.PsiEpsDef}
	\psi_\epsilon(x)=\frac{1}{\|\psi\|_1\epsilon^2}\psi\left(\frac{x}{\epsilon}\right),
\end{equation}
where $\|\cdot\|_1$ denotes the $L^1$-norm. The function $\psi_\epsilon$ is non-negative, smooth, has support contained in $B_{\epsilon/2}$, and satisfies $\|\psi_\epsilon\|_1=1$. Its Fourier transform has the property that, for every $N\in\N$,
\begin{equation}\label{eqn.PsiFourEst}
\widehat{\psi}_\epsilon(y)=\frac{\widehat{\psi}(\epsilon y)}{\|\psi\|_1}\ll_N\frac{1}{(1+|\epsilon y|^2)^N},\quad\text{for}~ y\in\R^2.
\end{equation}
These functions will be useful to us in our arguments below.

Now we consider indicator functions of balls and polygons in $\R^2$. For $R>0$, let $\chi_R$ denote the indicator function of $B_R\subseteq\R^2$. Then we have that
\[\widehat{\chi}_1(y)=\frac{J_1(2\pi|y|)}{|y|},\]
where $J_1$ is the Bessel function defined by
\[J_1(t)=\frac{1}{2\pi}\int_{0}^{2\pi} e^{i(t\sin \theta-\theta)}\mr{d}\theta.\]
Linearity properties of the Fourier transform, and standard estimates for Bessel functions (see \cite[Chapter IV, Lemma 3.11]{SteiWeis1971} and the paragraph preceding that lemma) then give that
\begin{equation}\label{eqn.ChiREstimate}
\widehat{\chi}_R(y)=\frac{RJ_1(2\pi R|y|)}{|y|}\ll \min\left\{R^2,\frac{R^{1/2}}{|y|^{3/2}}\right\}.
\end{equation}

If $W\subseteq\R^2$ then we write $\chi_W$ for the indicator function of $W$. The similarity in notation to $\chi_R$, as defined above, will not cause confusion in what follows. Observe that
\[\widehat{\chi}_W(y)=\int_We(-x\cdot y)\mr{d}x=\frac{-1}{(2\pi|y|)^2}\int_W\nabla_x^2\left( e(-x\cdot y)\right)\mr{d}x,\]
where $\nabla_x^2=\nabla_x\cdot\nabla_x$ is the Laplacian. By the Divergence Theorem, under reasonable hypotheses on $W$, this allows us to express the Fourier transform as a line integral over the boundary of $W$. Applying this argument in the case when $W$ is a polygon leads to a particularly nice formula for the Fourier transform, which we reproduce here in the form given in \cite{BranColzTrav1997}.
\begin{proposition}\label{prop.FourTranPoly}
	Let $P\subseteq\R^2$ be a polygon with counterclockwise oriented vertices $\{a_j\}_{j=1}^m$. Denote by $\sigma_j$ the unit vector parallel to the side $[a_j,a_{j+1}]$ and by $v_j$ the outward unit normal to this side. Then, with $a_{m+1}:=a_1$, we have
	\[\widehat{\chi}_P(y)=\frac{-1}{(2\pi|y|)^2}\sum_{j=1}^m\left(\frac{e(-y\cdot a_{j+1})-e(-y\cdot a_{j})}{y\cdot \sigma_j}\right)y\cdot v_j.\]
\end{proposition}
This proposition is proved as \cite[Lemma 2.2]{BranColzTrav1997}. The quantity in parenthesis in the sum is interpreted to equal
\[-\left(2\pi i |a_{j+1}-a_j|\right)e\left(\frac{-y\cdot(a_j+a_{j+1})}{2}\right),\]
when $y\cdot\sigma_j=0$. Proposition \ref{prop.FourTranPoly} implies that
\begin{equation}\label{eqn.PolyFourEst}
	\widehat{\chi}_P(y)\ll_P \frac{1}{|y|}\sum_{j=1}^m\min\left\{1,\frac{1}{|y\cdot\sigma_j|}\right\}.
\end{equation}
One thing to note is that these Fourier coefficients decay in most directions like $|y|^{-2}$. However, in directions perpendicular to the sides of $P$, they decay only like $|y|^{-1}$. This turns out to be a source of some difficulty, which is given special attention in the proofs below.

\section{Setup for counting estimates}\label{sec.Setup}

In light of the discussion at the end of Section \ref{sec.ModelSets}, we will now focus on the proof of Theorem \ref{thm.PenCPVersion}. Working within the scope of the CPS used in that section to define Penrose tilings, suppose that $\omega\in\C$ and $1\le m\le 4$ are chosen and write $W=W_{m,\omega}$. We wish to estimate, for $R\ge 2$,
\[N_R:=\#\left(\Lambda_{m,\omega}\cap B_R\right)=\sum_{(\lambda,\lambda^\star)\in\mc{L}}\chi_R(\lambda)\chi_W(\lambda^\star)=\sum_{\lambda\in L}\chi_R(\lambda)\chi_W(\lambda^\star).\]
In the last equation here we have used the injectivity of $\pi|_\cL$ to write our sum as a sum over $L=\pi(\cL).$ With a view to using the Poisson summation formula, we would like to smooth the summand by convolving the indicator functions with smooth bump functions. However, this slightly changes the supports of these functions, which non-trivially affects the main terms in our point counting formulas. To account for this, we will modify the boundaries of the regions involved, to bound our sum from above and from below by smoothed sums with main terms that are close together.

Suppose that $\epsilon, \delta<1/4$ are small positive numbers, to be chosen later as functions of $R$. Let $U=U(\partial W,\delta)$ denote the $\delta$-neighborhood of the boundary of $W$, and set
\[W^-=W\setminus U\quad\text{and}\quad W^+=W\cup U.\]
Then, with $\psi_{\delta}$ as in \eqref{eqn.PsiEpsDef}, define
\[\chi_{W,\delta}^\pm=\chi_{W^\pm}\ast\psi_{\delta},\]
so that
\[\chi_{W,\delta}^-\le \chi_W\le \chi_{W,\delta}^+.\]
Next let
\[\chi_{R,\epsilon}^\pm=\chi_{R\pm\epsilon}\ast\psi_\epsilon,\]
and define
\[N_{R,\epsilon,\delta}^\pm = \sum_{\lambda\in L}\chi_{R,\epsilon}^\pm(\lambda)\chi_{W,\delta}^\pm(\lambda^\star),\]
so that
\begin{equation}\label{eqn.Thm4PtCnt1}
N_{R,\epsilon,\delta}^-\le N_R\le N_{R,\epsilon,\delta}^+.
\end{equation}

By Proposition \ref{prop.PSF}, and using the injectivity of $\pi|_{\cL^*}$ (see the comment in Section \ref{sec.ModelSets} regarding this), we have that
\begin{align}
	N_{R,\epsilon,\delta}^\pm&=\mr{dens}(\cL)\sum_{\theta\in L^\circledast}\widehat{\chi}_{R,\epsilon}^\pm(\theta)\widehat{\chi}_{W,\delta}^\pm(\theta^\star)\nonumber\\
	&=\pi R^2\mr{dens}(\cL)|W|\nonumber\\
	&\hspace*{20bp}+O(\epsilon R+\delta R^2 + 1)\label{eqn.MainErr1}\\
	&\hspace*{20bp}+O\left(\sum_{\theta\in L^\circledast\setminus\{0\}}\widehat{\chi}_R(\theta)\widehat{\psi}_{\epsilon}(\theta)\widehat{\chi}_{W}(\theta^\star)\widehat{\psi}_\delta(\theta^\star)\right).\label{eqn.MainErr2}
\end{align}
In \eqref{eqn.MainErr2} we are also using the linearity of the Fourier transform to compare $\widehat{\chi}_{R\pm\epsilon}$ and $\widehat{\chi}_{W^\pm}$ with $\widehat{\chi}_{R}$ and $\widehat{\chi}_{W}$, respectively. 

In order to make the error term in \eqref{eqn.MainErr1} small, we must choose $\epsilon$ and $\delta$ small, as functions of $R$. However, this must be balanced by the fact that the error in \eqref{eqn.MainErr2} grows larger for small values of these parameters. The remainder of our argument involves analyzing the dependence of \eqref{eqn.MainErr2} on $\epsilon$ and $\delta$ to make these observations precise.

\section{Analysis of error terms: Proofs of Theorems \ref{thm.Pen} and \ref{thm.PenCPVersion}}\label{sec.PenProof}

Write $f(\theta,\theta^\star)$ for the summand in \eqref{eqn.MainErr2}. We divide the sum into three parts, as
\[\sum_{\theta\in L^\circledast\setminus\{0\}}f(\theta,\theta^\star)=\left(\sum_{0<|\theta|\le R^{-1}}+\sum_{R^{-1}<|\theta|\le 1}+\sum_{1<|\theta|}\right)f(\theta,\theta^\star)=E_1+E_2+E_3.\]

\noindent\emph{Bound for $E_1$:} For $E_1$ we use the bounds $\widehat{\psi}_\epsilon(\theta)\ll 1$ and $\widehat{\chi}_R(\theta)\ll R^2$ from \eqref{eqn.PsiFourEst} and \eqref{eqn.ChiREstimate}. Since $\cL^*$ is the Minkowski embedding of an order in the algebraic number field $K=\Q(\zeta_5)$, there is an integer $M$ (here one can take $M=25^2$) with the property that, for every $\theta\in L^\circledast$,
\begin{equation}\label{eqn.MChoice}
	\mathrm{Norm}_{K/\Q}(\theta)=\theta\theta^\star\in\frac{1}{M}\Z.
\end{equation}

Using these facts, we have that
\begin{align}
	E_1&\ll R^2\sum_{n=1}^\infty\sum_{\substack{0<|\theta|<R^{-1}\\\frac{nR}{M}\le |\theta^\star|<\frac{(n+1)R}{M}}}|\widehat{\chi}_{W}(\theta^\star)| |\widehat{\psi}_\delta(\theta^\star)|.\label{eqn.E1Bound1}
\end{align}
To estimate $\widehat{\chi}_W$ we will use the following lemma.
\begin{lemma}\label{lem.ChiWEst1}
	For any $(\theta,\theta^\star)\in \cL^*\setminus\{0\}$, we have that
	\[\widehat{\chi}_W(\theta^\star)\ll\frac{1}{|\theta^\star|}\min\left\{1,|\theta|\right\}.\]
\end{lemma}
\begin{proof}
	The bound $\widehat{\chi}_W(\theta^\star)\ll|\theta^\star|^{-1}$ follows immediately from \eqref{eqn.PolyFourEst}. To prove the other bound, we start with two observations:
	\begin{enumerate}[topsep=2pt,itemsep=4pt,parsep=4pt]
		\item[(i)] For any $(\theta,\theta^\star)\in\cL^*$ and for any integer $0\le j\le 4$, the point $(\zeta_5^j\theta,\zeta_5^{2j}\theta^\star)$ is an element of $\frac{1}{25}\cL$. This follows from the facts that $\cL^*\subseteq \frac{1}{25}\cL$ and $L=(1-\zeta_5)\Z[\zeta_5]$.
		\item[(ii)] If $(\lambda,\lambda^\star)\in\frac{1}{25}\cL$ then $(\overline{\lambda},\overline{\lambda^\star})\in\frac{1}{25}\cL$. This follows from the fact that complex conjugation commutes with the star map. One way to quickly see this is to note that $\mathrm{Gal}(K/\Q)$ is Abelian.
	\end{enumerate}
	Now suppose that $(\theta,\theta^\star)\in \cL^*\setminus\{0\}$, and choose $0\le j\le 4$ so that $|\mathrm{arg}(\zeta_5^{2j}\theta^\star)|\le \pi/5$. If we write $(\lambda,\lambda^\star)=(\zeta_5^j\theta,\zeta_5^{2j}\theta^\star)$ then it follows from \eqref{eqn.PolyFourEst}, together with the definition of $W$, that
	\begin{equation}\label{eqn.ChiWEst1}
		\widehat{\chi}_W(\theta^\star)\ll\frac{1}{|\lambda^\star|^2|\sin(\mathrm{arg}(\lambda^\star))|}.
	\end{equation}
	Since $(\overline{\lambda},\overline{\lambda^\star})\in\frac{1}{25}\cL$, we also have that
	\[(\lambda-\overline{\lambda},\lambda^\star-\overline{\lambda^\star})\in \frac{1}{25}\cL.\]
	Noting that $L=(1-\zeta_5)\Z[\zeta_5]$ does not contain any real numbers, the previous equation implies that
	\[\mr{Im}(\lambda)\mr{Im}(\lambda^\star)\gg 1.\]
	Then we have that
	\[|\lambda^\star|\sin(\mathrm{arg}(\lambda^\star))\gg |\mr{Im}(\lambda^\star)|\gg\frac{1}{|\lambda|},\]
	and this completes the proof.
\end{proof}
We pause here to comment that this lemma, and the other two lemmas in this section, are crucial steps in our argument. The proofs of these lemmas are delicate, and depend strongly on both the specific lattice and the precise shapes and orientations of the windows defining the cut and and project sets involved.

Now we apply Lemma \ref{lem.ChiWEst1} to \eqref{eqn.E1Bound1} to obtain
\begin{align}
	E_1&\ll R^2\sum_{n=1}^\infty\sum_{\substack{0<|\theta|<R^{-1}\\\frac{nR}{M}\le |\theta^\star|<\frac{(n+1)R}{M}}}\frac{|\theta|}{|\theta^\star|}\left(1+\delta^2|\theta^\star|^2\right)^{-N}\nonumber\\
	&\ll \sum_{n=1}^\infty\sum_{\substack{0<|\theta|<R^{-1}\\\frac{nR}{M}\le |\theta^\star|<\frac{(n+1)R}{M}}}\frac{1}{n}\left(1+\delta^2n^2R^2\right)^{-N}.\label{eqn.E1Bound2}
\end{align}
For a given choice of $n\in\N$, if we have two distinct points $(\theta_1,\theta_1^\star),(\theta_2,\theta_2^\star)\in\cL^*$ which satisfy
\[0<|\theta_i|<\frac{1}{R}\]
for $i=1, 2$, then, since their difference is in $\cL^*$, we must have that $|\theta_1^\star-\theta_2^\star|\gg R$. It follows that
\[\#\left\{\theta\in L^\circledast: |\theta|<\frac{1}{R},~ \frac{nR}{M}\le|\theta^\star|<\frac{(n+1)R}{M}\right\}\ll \frac{((n+1)R)^2-(nR)^2}{R^2}\ll n,\]
and using this in \eqref{eqn.E1Bound2} gives that
\begin{align}\label{eqn.E1Bound3}
	E_1\ll\sum_{n=1}^\infty(1+\delta^2n^2R^2)^{-N}\ll\int_{1}^\infty(1+\delta^2t^2R^2)^{-N}\mr{d}t\ll(\delta R)^{-1}.
\end{align}

\noindent\emph{Bound for $E_2$:} To estimate $E_2$, first write
\[\mc{S}_{m,n}=\left\{\theta\in L^\circledast :  \frac{2^{m-1}}{R}<|\theta|\le\frac{2^m}{R},~\frac{nR}{M2^m}\le|\theta^\star|<\frac{(n+1)R}{M2^m}\right\},\]
with $M$ chosen as in \eqref{eqn.MChoice}. Then we have that
\begin{align}
	E_2&\ll R^{1/2}\sum_{m=1}^{\lceil \log_2R\rceil}\sum_{n=1}^\infty\sum_{\theta\in\mc{S}_{m,n}}\frac{1}{|\theta|^{3/2}}|\widehat{\chi}_W(\theta^\star)|\left(1+\delta^2|\theta^\star|^2\right)^{-N}\nonumber\\
	&\ll R^2 \sum_{m=1}^{\lceil\log_2R\rceil}2^{-3m/2}\sum_{n=1}^\infty\left(1+\frac{\delta^2n^2R^2}{2^{2m}}\right)^{-N}\sum_{\theta\in\mc{S}_{m,n}}|\widehat{\chi}_W(\theta^\star)|.\label{eqn.E2Bound1}
\end{align}
Here we will use the following lemma.
\begin{lemma}\label{lem.ChiWEst2}
	With notation as above, for $m,n\in\N$ with $m\le\lceil \log_2R\rceil$, we have that
	\[\sum_{\theta\in\mc{S}_{m,n}}|\widehat{\chi}_W(\theta^\star)|\ll\frac{2^{2m}\log n}{nR^2}.\]
\end{lemma}
\begin{proof}
	Let us write
	\[\tilde{\mc{S}}=\left\{\lambda\in \frac{1}{25}L :  \frac{2^{m-1}}{R}<|\lambda|\le\frac{2^m}{R},~\frac{nR}{M2^m}\le|\lambda^\star|<\frac{(n+1)R}{M2^m},~|\mr{arg}(\lambda^\star)|\le \frac{\pi}{5}\right\}.\]
	By the argument used in the proof of Lemma \ref{lem.ChiWEst1} leading up to equation \eqref{eqn.ChiWEst1}, it is sufficient for this lemma to prove that
	\begin{equation*}
		\sum_{\lambda\in\tilde{\mc{S}}}\frac{1}{|\lambda^\star|^2|\sin(\mathrm{arg}(\lambda^\star))|}\ll\frac{2^{2m}\log n}{nR^2}.
	\end{equation*}
	Label the elements of $\star(\tilde{\mc{S}})$ as
	\[\lambda_j^\star=\rho_je^{it_j},\]
	for $1\le j\le J:=|\tilde{\mc{S}}|$, with each $\rho_j\in\R^+$ and with $0< |t_1|\le |t_2|\le\cdots \le |t_J|\le\pi/5$. Also, label the element of $\tilde{\mc{S}}$ corresponding to each $\lambda_j^\star$ as $\lambda_j$.
	
	By the conjugation argument used in the proof of Lemma \ref{lem.ChiWEst1}, we have that
	\[t_1\gg \frac{1}{|\lambda_1||\lambda_1^\star|}\gg\frac{1}{n}.\]
	Furthermore, by considering norms, we also have that
	\[|\lambda_{j+1}^\star-\lambda_{j}^\star|\gg \frac{1}{|\lambda_{j+1}-\lambda_j|}\gg\frac{R}{2^m}.\]
	Since the numbers $\lambda_j^\star$ are confined to an annulus of width $\ll R/2^m$, this forces
	\[t_j\gg\frac{j}{n},\]
	for all $1\le j\le J$, and it also shows that $J\ll n$. We therefore have that
	\begin{align*}
		\sum_{\lambda\in\tilde{\mc{S}}}\frac{1}{|\lambda^\star|^2|\sin(\mathrm{arg}(\lambda^\star))|}&\ll\frac{2^{2m}}{nR^2}\sum_{j=1}^J\frac{1}{j}\ll\frac{2^{2m}\log n}{nR^2},
	\end{align*}
	which completes the proof of the lemma.
\end{proof}
Substituting the bound from Lemma \ref{lem.ChiWEst2} into equation \eqref{eqn.E2Bound1}, we have that
\begin{align}
	E_2&\ll \sum_{m=1}^{\lceil\log_2R\rceil}2^{m/2}\sum_{n=1}^\infty\frac{\log n}{n}\left(1+\frac{\delta^2n^2R^2}{2^{2m}}\right)^{-N}\nonumber\\
	&\ll \sum_{m=1}^{\lceil\log_2R\rceil}2^{m/2}\int_1^\infty\frac{\log t}{t}\left(1+\frac{\delta^2t^2R^2}{2^{2m}}\right)^{-N}\mr{d}t\nonumber.
\end{align}
As long as we choose $\delta\gg R^{\kappa}$ for some fixed $\kappa>0$, we will obtain from this that
\begin{equation}\label{eqn.E2Bound2}
	E_2\ll_\kappa R^{1/2}\log R.
\end{equation}

\noindent\emph{Bound for $E_3$:} In order to estimate $E_3$, write
\begin{align*}
\mc{S}_{\ell,m,n}&=\left\{\theta\in L^\circledast : m<|\theta|\le m+1,~\frac{2\pi(\ell-1)}{m}\le\mr{arg}(\theta) <\frac{2\pi\ell}{m},\right.\\
&\hspace*{50bp} \left.\phantom{\frac{2}{m}}n-1<|\theta^\star|\le n\right\}.
\end{align*}
Here we will need the following result.
\begin{lemma}\label{lem.ChiWEst3}
	For $\ell,m,n\in\N$, we have that
	\[\sum_{\theta\in\mc{S}_{\ell,m,n}}|\widehat{\chi}_W(\theta^\star)|\ll\frac{\log n}{n}.\]
\end{lemma}
We will omit the proof of this lemma, which follows from essentially the same arguments used to prove Lemma \ref{lem.ChiWEst2}.

We now have that
\begin{align*}
	E_3&\ll\sum_{m=1}^{\infty}\sum_{n=1}^\infty\sum_{\ell=1}^m\sum_{\theta\in\mc{S}_{\ell,m,n}}\frac{R^{1/2}}{|\theta|^{3/2}}\left(1+\epsilon^2|\theta|^2\right)^{-N}|\widehat{\chi}_W(\theta^\star)|\left(1+\delta^2|\theta^\star|^2\right)^{-N}\nonumber\\
	&\ll R^{1/2}\sum_{m=1}^{\infty}\frac{\left(1+\epsilon^2m^2\right)^{-N}}{m^{3/2}}\sum_{n=1}^\infty\left(1+\delta^2n^2\right)^{-N}\sum_{\ell=1}^m\sum_{\theta\in\mc{S}_{\ell,m,n}}|\widehat{\chi}_W(\theta^\star)|\\
	&\ll R^{1/2}\sum_{m=1}^{\infty}\frac{\left(1+\epsilon^2m^2\right)^{-N}}{m^{1/2}}\sum_{n=1}^\infty\frac{\log n}{n}\left(1+\delta^2n^2\right)^{-N}\\
	&\ll R^{1/2}\sum_{m=1}^{\infty}\frac{\left(1+\epsilon^2m^2\right)^{-N}}{m^{1/2}}\sum_{n=1}^\infty\frac{\log n}{n}\left(1+\delta^2n^2\right)^{-N}.
\end{align*}
Comparing the series here with the corresponding integrals, again assuming that $\delta\gg R^{-\kappa}$, gives that
\begin{equation}\label{eqn.E3Bound1}
	E_3\ll_\kappa \epsilon^{-1/2}R^{1/2}\log R.
\end{equation}

\noindent\emph{Proof of Theorem \ref{thm.PenCPVersion}:} Combining our error estimates \eqref{eqn.E1Bound3}, \eqref{eqn.E2Bound2}, and \eqref{eqn.E3Bound1} for the $E_j$ sums, together with those in \eqref{eqn.MainErr1}, we have that
\[N_{R,\epsilon,\delta}^\pm-\pi R^2\mr{dens}(\cL)|W|\ll \epsilon R+\delta R^2 + (\delta R)^{-1}+R^{1/2}\log R+\epsilon^{-1/2}R^{1/2}\log R.\]
Choosing $\epsilon=(\log R)^{2/3}R^{-1/3}$ and $\delta=R^{-3/2}$ is optimal here, and produces the bound reported in Theorem \ref{thm.PenCPVersion}.

\noindent\emph{Proof of Theorem \ref{thm.Pen}:} Let $V$ be the vertex set of a unit length RPT. If $V$ can be constructed as a cut and project set, as $V=\Lambda_\omega$ for some $\omega\in\C$, then the result of Theorem \ref{thm.Pen} follows immediately from Theorem \ref{thm.PenCPVersion}. Otherwise, by Theorem \ref{thm.PenVerts}, there is a sequence of complex numbers $\{\omega_j\}_{j=1}^\infty$ with the property that, after a possible rotation,
\[V=\lim_{j\rar\infty}\Lambda_{\omega_j},\]
where the limit is taken in the Chabauty-Fell topology. For $R\ge 2$, the number of points in $\R^2$ which lie on the boundary of $B_R$, and which are the limits of sequences of points $\lambda_j\in\Lambda_{\omega_j}$, as $j\rar\infty,$ must be $\ll R^{2/3}(\log R)^{2/3}$. This follows from the uniformity over $\omega\in \C$ of the error term in Theorem \ref{thm.PenCPVersion}. It then follows, again by uniformity, that
\begin{align}
	\#(V\cap B_R)&=\lim_{j\rar\infty}\#\left(\Lambda_{\omega_j}\cap B_R\right)+O(R^{2/3}(\log R)^{2/3})\nonumber\\
	&=\pi R^2\mr{dens}(\cL)\sum_{m=1}^4|W_{m,\omega+m}|+O( R^{2/3}(\log R)^{2/3}).\label{eqn.ChabLim1}
\end{align}
Finally, we compute that
\[\mr{dens}(\cL)=\mr{disc}(L)^{-1/2}=\frac{4}{25\sqrt{5}}.\]
Substituting this in \eqref{eqn.ChabLim1}, together with the areas of the pentagonal windows, completes the proof of Theorem \ref{thm.Pen}.

%
%

\section*{Acknowledgements}
We would like to thank Michael Baake for several correspondences and for his detailed comments on the realization of Penrose tilings as cut and project sets.

\vspace{.15in}
		
{\footnotesize
\noindent
Department of Mathematics\\
University of Houston\\
Houston, TX, United States\\
haynes@math.uh.edu\\
clutsko@uh.edu
			
}


\begin{thebibliography}{1}
			

%

\bibitem{BaakGrim2013}
M. Baake and U. Grimm, {\it Aperiodic Order. Vol. 1: A Mathematical Invitation}, Cambridge Univ. Press, Cambridge (2013).

\vspace*{.1in}

\bibitem{BaakHayn2026}
M. Baake and A. Haynes, Convergence of Fourier-Bohr coefficients for regular Euclidean model sets, arXiv:2308.07105 (to appear in {\it Aperiodic Order. Vol. 3: Model Sets and Dynamical Systems}, eds. M. Baake, F. G\"ahler, and N. Ma\~nibo).

\vspace*{.1in}

\bibitem{Been1982}
F.~P.~M. Beenker, Algebraic theory of non periodic tilings of the plane by two simple building blocks: a square and a rhombus, TH Report 82-WSK-04 (1982), Technische Hogeschool, Eindhoven.

\vspace*{.1in}

\bibitem{BranColzTrav1997}
L. Brandolini, L. Colzani and G. Travaglini, Average decay of Fourier transforms and integer points in polyhedra, Ark. Mat. {\bf 35} (1997), no.~2, 253--275.

\vspace*{.1in}

\bibitem{deBr1981}
N.~G. de~Bruijn, Algebraic theory of Penrose's nonperiodic tilings of the plane. I, II, Nederl. Akad. Wetensch. Indag. Math. {\bf 43} (1981), no.~1, 39--52, 53--66.

\vspace*{.1in}

\bibitem{CassFroh1967}
J. W. S. Cassels and A. Fröhlich, Algebraic Number Theory, Academic Press, London (1967).

\vspace*{.1in}

\bibitem{GahlRhyn1986}
F. G\"ahler and J. Rhyner, Equivalence of the generalised grid and projection methods for the construction of quasiperiodic tilings, J. Phys. A {\bf 19} (1986), no.~2, 267--277.

\vspace*{.1in}

\bibitem{Gard1977}
M. Gardner, Mathematical Games, Scientific American, {\bf 236} (977), no.~1, 110--121.

\vspace*{.1in}

\bibitem{Gard1997}
M. Gardner, {\it Penrose tiles to trapdoor ciphers}, revised reprint of the 1989 original, 
MAA Spectrum, Math. Assoc. America, Washington, DC, 1997.

\vspace*{.1in}

\bibitem{GrunShep1987}
B. Gr\"unbaum and G.~C. Shephard, {\it Tilings and patterns}, Freeman, New York, 1987.

\vspace*{.1in}

\bibitem{Huxl2003}
M.~N.~Huxley, {\em Exponential sums and lattice points III},
Proc. London Math. Soc. (3)  {\bf 87}  (2003),  no. 3, 591--609.

\vspace*{.1in}

\bibitem{KoivLagacBjorHart2025}
H. Koivusalo and J. Lagac\'e (appendix by M. Bj\"orklund, T. Hartnick), Sharp density discrepancy for cut and project sets an approach via lattice point counting, Monatsh. Math. {\bf 208} (2025), no.~3, 397--444.

\vspace*{.1in}

\bibitem{Penr1974}
R. Penrose, The r\^{o}le of aesthetics in pure and applied mathematical research, Bull. Inst. Math. Appl. {\bf 10} (1974), 266--271.

\vspace*{.1in}

\bibitem{Penr1979}
R. Penrose, Pentaplexity: a class of nonperiodic tilings of the plane, Math. Intelligencer {\bf 2} (1979/80), no.~1, 32--37.

\vspace*{.1in}

\bibitem{Robi1996}
E.~A. Robinson Jr., The dynamical properties of Penrose tilings, Trans. Amer. Math. Soc. {\bf 348} (1996), no.~11, 4447--4464.

\vspace*{.1in}

\bibitem{RuhrSmilWeis2024}
R. R\"uhr, Y. Smilansky and B. Weiss, Classification and statistics of cut-and-project sets, J. Eur. Math. Soc. (JEMS) {\bf 26} (2024), no.~9, 3575--3638.

\vspace*{.1in}

\bibitem{Schi1972}
A.~Schinzel, {\em Wac\l{}aw Sierpi\'{n}ski's papers on the theory of numbers},
Acta Arith.  {\bf 21}  (1972), 7--13.

\vspace*{.1in}

\bibitem{SchmTrev2018}
S. Schmieding and R. Trevi\~no, Self affine Delone sets and deviation phenomena, Comm. Math. Phys. {\bf 357} (2018), no.~3, 1071--1112.

\vspace*{.1in}

\bibitem{Sene1995}
M.~W. Senechal, {\it Quasicrystals and geometry}, Cambridge Univ. Press, Cambridge, 1995.

\vspace*{.1in}

\bibitem{Sier1906}
W.~Sierpi\'{n}ski, {\em O pewnem zagadneniu w rachunku funkcyj asymptotznych}, Prace Mat.-Fiz. {\bf 17} (1906), 77--118.

\vspace*{.1in}

\bibitem{Solo2011}
Y. Solomon, {\em Substitution tilings and separated nets with similarities to the integer lattice}, Israel Journal of Mathematics {\bf 181} (2011), 445--460.

\vspace*{.1in}

\bibitem{SteiWeis1971}
E.~M. Stein and G.~L. Weiss, {\it Introduction to Fourier analysis on Euclidean spaces}, Princeton Mathematical Series, No. 32, Princeton Univ. Press, Princeton, NJ, 1971.


\vspace*{.1in}

\bibitem{Swinn2001}
H.~P.~F. Swinnerton-Dyer, {\it A brief guide to algebraic number theory}, London Mathematical Society Student Texts, 50, Cambridge Univ. Press, Cambridge, 2001.

\end{thebibliography}
	\end{document}